\documentclass[12pt]{amsart}

\usepackage[hang,flushmargin]{footmisc}

\usepackage{a4wide}

\usepackage{amssymb}
\usepackage{amsmath}
\usepackage{amsthm}
\usepackage[foot]{amsaddr}

\usepackage{subfigure}
\usepackage{graphicx}
\usepackage{caption}
\usepackage[utf8]{inputenc}
\usepackage{amsfonts}

\usepackage{hyperref}
\usepackage{xcolor}

\usepackage{verbatim}

\usepackage{mathtools}
\mathtoolsset{showonlyrefs}

\numberwithin{equation}{section}
\newtheorem{theorem}{Theorem}[section]

\newtheorem{lemma}[theorem]{Lemma}

\newtheorem{remark}[theorem]{Remark}\newtheorem{proposition}[theorem]{Proposition}\newtheorem{definition}[theorem]{Definition}

\newtheorem*{theorem*}{Theorem}

\renewcommand{\L}{\Lambda}
\renewcommand{\l}{\lambda}
\newcommand{\R}{{\mathbb R}}  \newcommand{\Z}{{\mathbb Z}} \newcommand{\N}{{\mathbb N}}

\newcommand{\C}{{\mathbb C}}   

\newcommand{\vol}{{\rm vol}}

\providecommand{\norm}[1]{\lVert#1\rVert}

\title[On the frame property of Hermite functions]{On the frame property of Hermite functions and exploration of their frame sets}

\author[M.~Faulhuber]{Markus Faulhuber$^{*}$}
\address{$^{*}$Faculty of Mathematics, University of Vienna \newline Oskar-Morgenstern-Platz 1, 1090 Vienna, Austria
\newline $^{\dagger}$ Department of Mathematical Sciences, Norwegian University of Science and Technology (NTNU), 7491 Trondheim, Norway
}
\email{markus.faulhuber@univie.ac.at}

\author[I.~Shafkulovska]{Irina Shafkulovska$^{*}$}
\email{irina.shafkulovska@univie.ac.at}

\author[I.~Zlotnikov]{Ilya Zlotnikov$^{*, \dagger}$}
\email{ilia.k.zlotnikov@ntnu.no}

\thanks{This research was funded in whole or in part by the Austrian Science Fund (FWF) [\href{https://doi.org/10.55776/P33217}{10.55776/P33217}].  I.~Z. also gratefully acknowledges support from the Research Council of Norway by Grant 334466. I.~Z.\ is grateful to I.~Krasikov for fruitful discussions and for providing the full text of \cite{MR2168916}. The authors gratefully acknowledge valuable feedback from the anonymous referees.}



\begin{document}
\begin{abstract}
    We study Gabor frames with Hermite window functions. Gröchenig and Lyubarskii provided a sufficient density condition for their frame sets, which leads to what we call the ``safety region". For rectangular lattices and Hermite windows of order 4 and higher, we enlarge this safety region by providing new points on the boundary of this region. For this purpose, we employ the Janssen representation of the frame operator to compare its distance to the identity in the operator norm. The calculations lead to estimates on series involving Laguerre polynomials with Gaussian weight functions.
\end{abstract}

\vspace*{-1.8cm}
\maketitle
\vspace*{-.6cm}

\medskip
\medskip
\hrule
\smallskip
\begin{center}
	Dedicated to Karlheinz Gröchenig on the occasion of his 65th birthday.
\end{center}
\smallskip
\hrule
\smallskip
    
\section{Introduction and notation}
\subsection{Introduction}

The study of frame sets has seen breakthroughs and surprises in recent years. For Hermite functions, there was hope that the structure of their frame sets may be simple and can be fully understood soon. However, \textit{The Mystery of Gabor frames}, after the eponymous article \cite{Gro14} by Gröchenig, is an appropriate description of our understanding of Gabor systems with Hermite functions. Our results contribute to this puzzle.

In this work, we are concerned with frame sets of Hermite functions. This means that we seek to find rectangular lattices $\L \subset \R^2$ such that the resulting Gabor system is a frame. The Hermite function $h_n$ of order $n$, scaled to unit norm, is (cf.~\cite[Chap.~1.7]{Fol89}, \cite{HorLemVid23})
\begin{equation}
    h_n(t) =  c_n (-1)^n e^{\pi t^2} \frac{d^n}{dt^n}e^{-2\pi t^2}, \, t \in \R,
    \qquad \text{so } \qquad
    \norm{h_n}_{L^2(\R)} = 1, \, n \in \N_0,
\end{equation}
where $c_n = 2^{1/4}/\sqrt{n! (2\pi)^n 2^n}$. In particular, the $n$-th Hermite function is a Gaussian function multiplied with the $n$-th Hermite polynomial and $n=0$ gives the Gaussian function
\begin{equation}
    h_0(t) = 2^{1/4} e^{-\pi t^2}.
\end{equation}
Gabor systems with Hermite window functions were studied, e.g., by Gröchenig and Lyubarskii \cite{GroLyu07}, \cite{GroLyu09}. We recall one of the main results from \cite{GroLyu07}.
\begin{theorem*}[Gröchenig, Lyubarskii]
    Consider the Gabor system $\mathcal{G}(h_n,\L)$. If the density of the lattice $\L$ exceeds $n+1$, then $\mathcal{G}(h_n,\L)$ is a Gabor frame for $L^2(\R)$.
\end{theorem*}
\noindent
Existing results, which we will explain shortly, show in particular the following:
\begin{center}
    the Gabor system $\mathcal{G}_n = \mathcal{G}(h_n, (1/\sqrt{n+1}) \Z^2)$ is \textit{not} a frame for $n=0,1,2,3$.
\end{center}
It seems natural to expect that this pattern continues, but, again, the frame set of Hermite functions keeps being mysterious. This is underlined by our upcoming main result.

\begin{theorem}[{\bf Main result}]\label{thm:main}
    Denote by $\mathcal{G}_n = \mathcal{G}(h_n, (1/\sqrt{n+1}) \Z^2)$ the Gabor system for the $n$-th Hermite function and square lattice of density $n+1$. Then $\mathcal{G}_n$ is a Gabor frame for $L^2(\R)$ if and only if $n \geq 4$.
\end{theorem}
As Theorem \ref{thm:main} shows, the opposite of the expected pattern holds true. So, at least for square lattices, equality can be achieved in the Theorem of Gröchenig and Lyubarskii once $n \geq 4$. The case $n=4$ is also the first Gabor system of the form
\begin{equation}\label{eq:Gn}
    \mathcal{G}_n = \mathcal{G}(h_n, (1/\sqrt{n+1}) \Z^2) = \{ \pi(\l) h_n \mid \l \in (1/\sqrt{n+1}) \Z^2 \},
\end{equation}
which has not been treated in the literature so far ($\pi(\l)$ is defined below in \eqref{eq:def_TFshift}). We will restrict our attention to the frame set of $h_n$ for rectangular lattices of the form $\L = a \Z \times b \Z$.

The case of the Gaussian function, i.e., $n=0$, has been completely settled (for general point sets) independently by Lyubarskii \cite{Lyu92} and in the articles of Seip \cite{Sei92} and Seip and Wallst\'en \cite{SeiWal92}. The case of the first Hermite function (and more generally odd functions) was treated by Lyubarskii and Nes \cite{LyuNes13} for rectangular lattices and extended to general lattices by one of the authors in \cite{Fau20}.

In \cite{Gro14}, Gröchenig came up with the following questions, which are now called the frame set conjecture for Hermite functions:

\medskip
$\bullet$ For $h_n$ with $n \in 2\N$, is the Gabor system $\mathcal{G}(h_n, \L)$ always a frame when $\vol(\L)<1$?

\medskip

$\bullet$ For $h_n$ with $n \in 2\N-1$, is the Gabor system $\mathcal{G}(h_n, \L)$ always a frame when $\vol(\L)<1$ and $\vol(\L) \neq \frac{N}{N+1}$, $N \in \N$?

\medskip

The questions were originally posed for rectangular lattices and are based on the restriction found in \cite{LyuNes13}. Due to the extension given in \cite{Fau20} they may be formulated for arbitrary lattices. However, the frame set conjecture for Hermite functions was answered negatively by Lemvig \cite{Lem16}, who provided first counter-examples. More exceptions by Horst, Lemvig, and Videbaek \cite{HorLemVid23} followed. A sufficiency criterion with numerical investigations was presented by Ghosh and Selvan \cite{GosSel23}.

This work is structured as follows. In Sec.~\ref{sec:notation} we introduce the general notation used in this work. In Sec.~\ref{sec:pre}, we recall essential results from Gabor analysis in Sec.~\ref{sec:Gabor}, provide necessary details on Laguerre polynomials in Sec.~\ref{sec:Laguerre} and introduce auxiliary results and outline the rough proof strategy in Sec.~\ref{sec:aux}. Sec.~\ref{sec:results} is devoted to new results on the frame set of Hermite functions. Our main result, Theorem \ref{thm:main}, is proved in Sec.~\ref{sec:proof}. The limitations of our approach are discussed in Sec.~\ref{sec:failure}. Finally, in Sec. \ref{sec:outside}, we present more applications of the approach we used to prove Theorem \ref{thm:main}. We enlarge what we call the safety region for $h_{15}$ and discover a new neighborhood in the frame set of $h_9$.

\subsection{Dedication}
All authors are particularly pleased to dedicate this paper to Karlheinz Gröchenig. M.~Faulhuber is grateful for Karlheinz' mentorship and invaluable advice throughout the last decade. I.~Shafkulovska expresses her sincere gratitude to Karlheinz for his guidance and enjoyable, instructive collaboration. I. Zlotnikov is grateful to Karlheinz for many fruitful discussions on various topics in mathematics during the postdoctoral fellowship at the University of Vienna.

\subsection{Notation}\label{sec:notation}
We mainly follow the notation from the textbook of Gröchenig \cite{Gro01}. In this work, as usual in time-frequency analysis, we use the following Fourier transform of a suitable function on $\R$:
\begin{equation}
    \mathcal{F} f(\omega) = \int_{\R} f(t) e^{-2 \pi i \omega t} \, dt.
\end{equation}
This extends to a unitary operator on $L^2(\R)$, hence we have Plancherel's theorem and Parseval's identity, respectively, given by
\begin{equation}
    \norm{f}_{L^2(\R)} = \norm{\mathcal{F} f}_{L^2(\R)}
    \quad \text{ and } \quad
    \langle f, g \rangle = \langle \mathcal{F} f, \mathcal{F} g \rangle.
\end{equation}
The inner product and norm on $L^2(\R)$ are, respectively, given by
\begin{equation}
    \langle f, g \rangle = \int_\R f(t) \overline{g(t)} \, dt
    \quad \text{ and } \quad
    \norm{f}_{L^2(\R)}^2 = \langle f, f \rangle. 
\end{equation}
Sometimes, we will pass an index to the inner product, denoting the specific Hilbert space, which will be either $L^2(\R)$ or $L^2(\R^2)$. Note that the Hermite functions are eigenfunctions of the Fourier transform, in particular we have
\begin{equation}\label{eq:Fhn}
    \mathcal{F} h_n = (-i)^n h_n.
\end{equation}

A Gabor system is a set of functions of the form
\begin{equation}
    \mathcal{G}(g, \L) = \{ \pi(\l) g \mid \l \in \L \},
\end{equation}
where $g \in L^2(\R)$ (of course, not the zero function throughout this work), $\L \subset \R^2$ discrete and $\pi(z)$ is a time-frequency shift by $z$, which is a unitary operator, acting by the rule
\begin{equation}\label{eq:def_TFshift}
    \pi(z) g = M_\omega T_x g, \quad z =(x,\omega) \in \R^2.
\end{equation}
Here, $T_x$ and $M_\omega$ denote the usual operators of translation (time-shift) and modulation (frequency-shift), which act on functions by the rules
\begin{equation}
    T_x g(t) = g(t-x)
    \qquad \text{ and } \qquad
    M_\omega g(t) = e^{2 \pi i \omega t} g(t),
\end{equation}
with the commutation relation
\begin{equation}\label{eq:CR}
    M_\omega T_x = e^{2 \pi i \omega x} T_x M_\omega.
\end{equation}
Often $\L$ is assumed to be a lattice in the time-frequency plane, i.e., a discrete co-compact subgroup of $\R^2$. We denote the co-volume of $\L$ simply by $\vol(\L)$ in this work and will actually only deal with the case of rectangular lattices, i.e., $\L = \L_{a,b} = a \Z \times b \Z$, with $a,b \in \R_+$. We will also identify the lattice $\L_{a,b}$ with the pair of parameters $(a,b) \in \R_+ \times \R_+$. Note that for a rectangular lattice we have
\begin{equation}
    \vol(\L_{a,b}) = ab
    \quad \text{ and its density is } \quad
    \delta = \vol(\L)^{-1} = 1/(ab).
\end{equation}
For a general lattice $\L = M \Z^2$, $M \in GL(2,\R)$, we have $\vol(\L) = |\det(M)|$. For the lattice $\Lambda$ we denote by $\Lambda^\circ$ its adjoint lattice, which is $J \L^\perp$, where $\L^\perp$ is the classical dual lattice and $J$ the standard symplectic matrix;
\begin{equation}
    J =
    \begin{pmatrix}
        0 & 1\\
        -1 & 0
    \end{pmatrix} .
\end{equation}
Note that while the classical dual lattice is defined by the condition
\begin{equation}
    \L^\perp = \{ \l^\perp \mid \l^\perp \cdot \l \in \Z \ \forall \l \in \L\},
\end{equation}
the adjoint lattice may be defined by the commutation relations of time-frequency shifts
\begin{equation}
    \L^\circ = \{ \l^\circ \mid \pi(\l) \pi(\l^\circ) = \pi(\l^\circ) \pi(\l) \ \forall \l \in \L \}.
\end{equation}
This is, by \eqref{eq:CR}, equivalent to saying that $\l^\circ \in \L^\circ$ if and only if $\sigma(\l^\circ,\l) \in \Z$, for all $\l \in \L$, where $\sigma(\l^\circ , \l) = \l^\circ \cdot J \l$ denotes the standard symplectic form. Note that for a rectangular lattice $\L_{a,b} = a \Z \times b \Z$ the adjoint lattice is given by $\L_{a,b}^\circ = (1/b) \Z \times (1/a) \Z$.

The adjoint (or symplectic dual) lattice makes the usage of a symplectic Poisson summation formula possible. For a (suitable) function $F$ on $\R^2$, this is
\begin{equation}\label{eq:PSF}
    \sum_{\l \in \L} F(\l + z) = \vol(\L)^{-1} \sum_{\l^\circ  \in \L^{\circ}} \mathcal{F}_\sigma F(\l^\circ) e^{2 \pi i \sigma(z,\l^\circ)},
\end{equation}
where $\mathcal{F}_\sigma$ is the symplectic Fourier transform, see, e.g., \cite{Gos17},
\begin{equation}
    \mathcal{F}_\sigma F(z) = \iint_{\R^2} F(\xi) e^{-2 \pi i \sigma(\xi,z)} \, d\xi.
\end{equation}
The symplectic Poisson summation formula \eqref{eq:PSF} is implicitly present throughout the whole article. We mainly apply it to the short-time Fourier Transform (STFT) of Hermite functions. 
The STFT of a function $f \in L^2(\R)$ with window $g \in L^2(\R)$ is given by
\begin{equation}
    V_g f(x,\omega) = \int\limits_{\R} f(t) \overline{g(t-x)} e^{-2\pi i \omega t} dt = \langle f, \pi(z) g \rangle, \quad z = (x,\omega) \in \R^2.
\end{equation}
A Gabor system $\mathcal{G}(g,\L)$ is a frame for $L^2(\R)$ if and only if there exist positive constants $0<A\leq B<\infty$, such that the frame inequality holds:
\begin{equation}\label{eq:frame}
    A \norm{f}_{L^2(\R)}^2 \leq \sum_{\l \in \L} | \langle f , \pi(\l) g \rangle |^2 \leq B \norm{f}_{L^2(\R)}^2, \quad \forall f \in L^2(\R).
\end{equation}
The constants $A$ and $B$ are called lower and upper frame bound, respectively, and the sharpest possible bounds are the spectral bounds of the associated Gabor frame operator $S_{g,\L}$, which acts on $f \in L^2(\R)$ by
\begin{equation}
    S_{g,\L} f = \sum_{\l \in \L} \langle f, \pi(\l) g \rangle \pi(\l) g.
\end{equation}
The set of all (rectangular) lattices such that the Gabor system $\mathcal{G}(g,\L)$ is a Gabor frame for $L^2(\R)$, is called the frame set $\mathfrak{F}$ of $g$:
\begin{equation}
    \mathfrak{F}(g) = \{ \L \subset \R^2 \mid \mathcal{G}(g,\L) \text{ is a frame} \}.
\end{equation}
This is also the terminology used by Gröchenig in \cite{Gro14}. An important result for functions in the modulation space $M^1(\R)$, defined in Def.~\ref{def:Mp} just below, is that the frame set is an open subset of $\{\L \mid \vol(\L) < 1\}$ and that it contains a neighborhood of 0 \cite[Thm.~2.1]{Gro14} (see also \cite{FeiKai04}, \cite{FeiKoz98}).
The latter statement guarantees that Gabor frames with a window in $M^1(\R)$ exist if the density of $\L$ is large enough. For rectangular lattices $\L_{a,b} = a \Z \times b \Z$, we can picture the frame set by identifying a lattice with the pair of parameters $(a,b) \in \R_+ \times \R_+$, as done in Fig.~\ref{fig:frame_set}. Due to \eqref{eq:Fhn}, the frame set for Hermite functions is symmetric in the sense that $(a,b) \in \mathfrak{F}(h_n)$ if and only if $(b,a) \in \mathfrak{F}(h_n)$.

\begin{figure}[h!t]
    \centering
    \includegraphics[width=.75\textwidth]{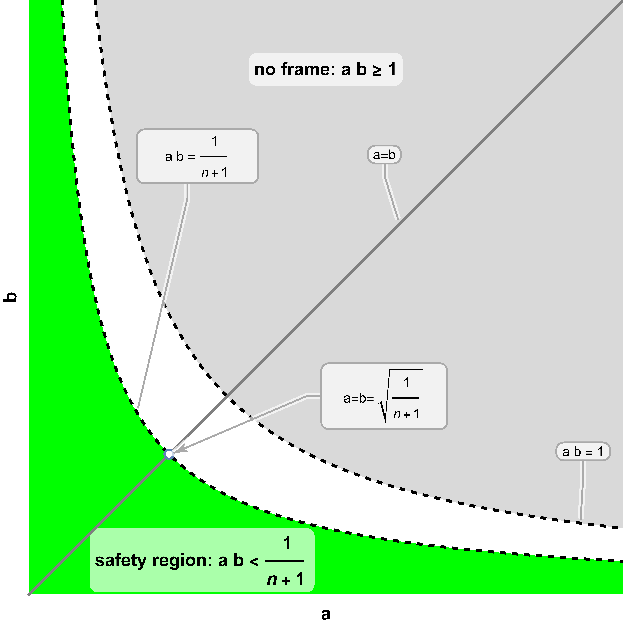}
    \caption{For the $n$-th Hermite function we know that the Gabor system $\mathcal{G}(h_n,a \Z \times b \Z)$ is a frame if $ab < 1/(n+1)$ and that it is not a frame if $a b \geq 1$. The frame set structure of $\mathfrak{F}(h_n)$, $n\geq 1$, between the ``safety region" and the ``no frame" region is currently only partially understood.}
    \label{fig:frame_set}
\end{figure}
For the $n$-th Hermite function we call the part of the frame set $\vol(\L) < 1 /(n+1)$ the safety region, as we have a frame by the theorem of Gröchenig and Lyubarskii \cite{GroLyu07}. For rectangular lattices $\L_{a,b} = a \Z \times b \Z$, $a,b \in \R_+$, the safety region in $\mathfrak{F}(h_n)$ consists of all pairs of parameters such that $ab < 1/(n+1)$.

\section{Preliminaries}\label{sec:pre}
\subsection{Results from Gabor analysis}\label{sec:Gabor}
We recall some well-established results from Gabor analysis. The modulation space $M^1(\R)$, which is also called Feichtinger's algebra due to its discoverer Feichtinger \cite{Fei81}, will provide a convenient set-up for our purposes. There are many equivalent definitions of this space and we refer to the treatise by Jakobsen~\cite{Jak18} for details. For the definition of the modulation spaces one needs to fix a window in the Schwartz space $\mathcal{S}(\R)$, but all windows lead to the same space with equivalent norms. Hence, we will pick the Gaussian window $h_0$ in our definition. More details can be found in the textbook of Gröchenig \cite[Chap.~11--12]{Gro01}.
\begin{definition}[Modulation space]\label{def:Mp}
    The modulation space $M^p(\R)$ consists of all tempered distributions $f \in \mathcal{S}'(\R)$ such that
    \begin{equation}
        \norm{f}_{M^p(\R)} = \norm{V_{h_0}f}_{L^p(\R^2)} < \infty.
    \end{equation}
\end{definition}
Note that $V_g g \in M^1(\R^2)$ if $g \in M^1(\R)$ \cite[Cor.~5.5]{Jak18}. Hence, Feichtinger's algebra is a convenient setting for us, as it contains $h_n$. Therefore, all manipulations in the sequel are justified and we do not run into any convergence issues (all series converge unconditionally).

The following result was first shown by Janssen \cite{Jan98} for rectangular lattices and for general lattices by Feichtinger and Luef \cite{FeiLue06}. We also refer to the work of Gröchenig and Koppensteiner \cite{GroKop19}, where all technical details are clarified.
\begin{proposition}[Fundamental identity of Gabor analysis]
    Assume we have functions $f_1, f_2, g_1, g_2 \in M^1(\R)$.
    Then
    \begin{equation}\label{eq:FIGA}
        \sum_{\l \in \L} V_{g_1} f_1(\l) \, \overline{V_{g_2} f_2 (\l)} = \vol(\L)^{-1} \sum_{\l^\circ \in \L^\circ} V_{g_1} g_2(\l^\circ) \, \overline{V_{f_1} f_2(\l^\circ)}.
    \end{equation}
\end{proposition}
The proof uses symplectic Poisson summation \eqref{eq:PSF} and the orthogonality relations
\begin{equation}\label{eq:OR}
         \langle V_{g_1} f_1, V_{g_2} f_2 \rangle_{L^2(\R^2)} = \langle f_1, f_2 \rangle_{L^2(\R)} \overline{\langle g_1, g_2 \rangle}_{L^2(\R)}.
     \end{equation}
The latter relation is the STFT-version of Moyal's identity for the Wigner distribution introduced in \cite{Moy49} (see also \cite[Chap.~6]{Gos17}, \cite[Chap.~4.3]{Gro01}).
We continue with the following observation, which is a slight variation of \cite[Lem.~4.2.1]{Gro01} in the textbook of Gröchenig. The proof employs the Cauchy-Schwarz inequality, in particular, the fact that equality can only be achieved for linearly dependent vectors.
\begin{lemma}\label{lem:Vgg}
    For $g \in L^2(\R)$ we have $V_g g(0,0) \in \R$. Moreover, we have
    \begin{equation}
        V_g g(0,0) = \norm{g}_{L^2(\R)}^2
    \end{equation}
    and the following estimate for all $(x,\omega) \in \R^2 \setminus \{(0,0)\}$:
    \begin{equation}
        |V_g g(x,\omega)| < V_g g(0,0).
    \end{equation}
\end{lemma}
The following notion was introduced in \cite{TolOrr95}. A function $g$ satisfies \textit{Condition A} if
\begin{equation}\label{eq:cond_A}
    \sum_{\l^\circ \in \L^\circ} | V_g g (\l^\circ) | < \infty. \tag{Condition A}
\end{equation}
The next result was shown by Tolimieri and Orr \cite{TolOrr95} under the assumption that the window function $g$ indeed satisfies \eqref{eq:cond_A}. We only add the convenient choice that we have $g \in M^1(\R)$, so that we can directly use the result.
\begin{lemma}
    Let $g \in M^1(\R)$, then it satisfies \eqref{eq:cond_A} and, consequently, the Gabor system $\mathcal{G}(g,\L)$ has a finite upper frame bound $B$, which is at most
    \begin{equation}\label{eq:B}
        B \leq \vol(\L)^{-1} \sum_{\l^\circ \in \L^\circ} |V_g g(\l^\circ)| < \infty.
    \end{equation}
\end{lemma}
The proof of the previous result follows easily from \eqref{eq:FIGA} and Lemma \ref{lem:Vgg}.

\begin{lemma}[Janssen representation]
    For a window $g \in M^1(\R)$ and a lattice $\L$, the frame operator $S = S_{g,\L}$ can be expressed as
    \begin{equation}
        S = \vol(\L)^{-1} \sum_{\l^\circ \in \L^\circ} \langle g, \pi(\l^\circ) g \rangle \pi(\l^\circ).
    \end{equation}
\end{lemma}
The result is well-known and it follows basically from \eqref{eq:FIGA}.
The last result in this section is the following sufficiency criterion from \cite{Wie13_PhD}, also known as the Janssen test. As this will be the basis for all our proofs in Sec.~\ref{sec:results}, we provide the proof for this result.
\begin{lemma}[Janssen test]\label{lem:Janssen_test}
    Let $\Lambda$ be a lattice and $g$ a window with $\norm{g}_{L^2(\R)} = 1$. If 
    \begin{equation}\label{eq:Janssen_test}
        \sum\limits_{\lambda^\circ \in \Lambda^\circ} |V_g g(\lambda^o)| < 2,
    \end{equation}
    then $\mathcal{G}(g, \Lambda)$ constitutes a Gabor frame for $L^2(\R)$.
\end{lemma}
\begin{proof}
    We use the proof given in \cite{Wie13_PhD} and compare the scaled frame operator $\vol(\L) S$, with $g \in M^1(\R)$ and $\L \subset \R^2$ a lattice, to the identity operator $I = I_{L^2(\R) \to L^2(\R)}$. Noting that $\pi(0) = I$, and using the Janssen representation of $S$, we get
    \begin{align}
        \norm{\vol(\L) S - I}_{L^2(\R) \to L^2(\R)}
        & = \norm{\sum_{\l^\circ \in \L^\circ \setminus \{ 0 \}} \langle g, \pi(\l^\circ) g \rangle \pi(\l^\circ)}_{L^2(\R) \to L^2(\R)}\\
        & \leq \sum_{\l^\circ \in \L^\circ \setminus \{0\}} |  \langle g, \pi(\l^\circ) g \rangle | = \widetilde{B}<1.
    \end{align}
    By our assumption \eqref{eq:Janssen_test}, we see that $\vol(\L) S$ is less than 1 away from the identity operator. Therefore we can build the Neumann series of the inverse and estimate its norm by
    \begin{equation}
        \norm{(\vol(\L)S)^{-1}}_{L^2(\R) \to L^2(\R)} \leq \sum_{k=0}^\infty \widetilde{B}^k = \frac{1}{1-\widetilde{B}},
    \end{equation}
    which shows that $\norm{(\vol(\L)S)^{-1}} \leq 1/(1-\widetilde{B})$ and so the frame operator is invertible. It follows that the Gabor system has a lower frame bound. The upper frame bound is finite due to \eqref{eq:B}, hence $\mathcal{G}(g,\L)$ is a frame.
\end{proof}

To the best of our knowledge, this tool has so far not been used in rigorous analytic investigations of frame sets. It was introduced in \cite{Tsc00_Master} for finite Gabor systems in $\C^L$, $L \in \N$ and then presented in \cite[Chap.~3.1, Cor.~3.1]{Wie13_PhD} for Gabor systems in $L^2(\R)$. In both cases it was combined with numerical investigations to have heuristic arguments for certain optimal sampling strategies. It was again discussed in \cite[Sec.~3]{Fau21-AA} for heuristic reasons and in \cite[Sec.~6]{Skr21} an extension of the result was presented for Gabor $g$-frames. It was called ``Janssen condition" in \cite[Chap.~4.1]{Tsc00_Master} as the criterion arises from the Janssen representation of the frame operator and we also refer to \cite[Sec.~6]{Fei19} for further reading.

\subsection{Laguerre polynomials}\label{sec:Laguerre}
Recall that the $n$-th Laguerre polynomial, denoted by $\mathcal{L}_n$, is defined on the positive real axis $\R_+$ (but can be extended to $\R$) by
\begin{equation}
    \mathcal{L}_n(x) = \sum_{k=0}^n \binom{n}{k} \frac{(-x)^k}{k!}.
\end{equation}
The Laguerre polynomials are orthogonal on $\R_+$ with respect to the exponential weight $w(x) = e^{-x}$ \cite[Sec.~5.1, eq.~(5.1.1)]{Sze39} and $\mathcal{L}_n(x)$ has $n$ distinct roots in $\R_+$ \cite[Chap.~3.3, Chap.~6.31]{Sze39}. The largest root $x_n$ of $\mathcal{L}_n$ satisfies \cite[Thm.6.31.2]{Sze39}
\begin{equation}
    x_n < 2n+1+\sqrt{(2n+1)^2+1/4},
    \quad \text{ and } \quad
    x_n \sim 4n.
\end{equation}

Our interest in Laguerre polynomials comes from the following result, which may be looked up in \cite[Chap.~1.9, Thm.~1.104]{Fol89}. Note that $V_gf$ is defined slightly differently in \cite[Chap.~1.8, eq.~(1.88)]{Fol89} (where it is more or less the ambiguity function), which makes the appearance of a phase factor explicit in our work.
\begin{lemma}\label{lem:Vh_L}
    For the $n$-th Hermite function $h_n$, we have 
    \begin{equation}\label{eq:Vh_L}
        V_{h_n} h_n (x,\omega) = e^{-\pi i x \omega} \mathcal{L}_n\left(\pi (x^2+\omega^2)\right) e^{-\pi(x^2+\omega^2)/2}.
    \end{equation}
\end{lemma}
In the sequel, our goal will always be to find $\L$ such that
\begin{equation}
    \sum_{\l^\circ \in \L^\circ} |V_{h_n} h_n(\l^\circ)| = \sum_{\l^\circ \in \L^\circ} \left|\mathcal{L}_n \left( \pi |\l^\circ|^2 \right) \right| e^{-\pi |\l^\circ|^2/2} < 2.
\end{equation}
For square lattices of density $\delta$, this becomes explicitly \eqref{eq:explicit_square_janssen} below.

\subsection{Auxiliary results and proof strategy}\label{sec:aux}
In this section, we present some auxiliary results and present the general proof strategy.

Upper bounds for Laguerre polynomials were investigated in a number of papers, see, e.g., \cite{MR2168916}, \cite{MR2287374}, \cite{MR1662719}. There is a classical estimate, known as Szeg\H{o} bound, which is
\begin{equation}\label{eq:Szego}
    |\mathcal{L}_n(x)| \le e^{x/2} \quad \text{ for } x \ge 0,
\end{equation}
proved by Szeg\H{o} \cite[Sec.~7.21]{Sze39}. However, \eqref{eq:Szego} will not be sufficient for our purposes as, using \eqref{eq:Vh_L}, it only gives that $|V_{h_n} h_n|$ is bounded by 1, which we also know from Lemma~\ref{lem:Vgg}. It is worth mentioning that, in the other direction, Lemma \ref{lem:Vgg} in combination with \eqref{eq:Vh_L} implies the Szeg\H{o} bound. We need more delicate estimates in the sequel.

First, observe that it is easy to obtain a crude polynomial bound for $\mathcal{L}_n$, as we have
\begin{equation}\label{eq:bound_xn}
    |\mathcal{L}_n(x)| \leq \sum_{k = 0}^n \binom{n}{k} \frac{x^k}{k!} \leq (n+1) \binom{n}{\lfloor n/2 \rfloor} x^n \leq (n+1) \, 2^n x^n, \quad \forall x \geq 1.
\end{equation}
Here, the notation ${\lfloor m \rfloor}$ stands for the greatest integer less than or equal to $m$. Note that the binomial coefficient may be huge, but this estimate becomes useful once $x$ is, say, larger than $x_n$, the largest root of $\mathcal{L}_n$.
Our estimate~\eqref{eq:bound_xn} will be useful in the calculations to prove Theorem~\ref{thm:main} for $n$ sufficiently small, which in our case is $n \leq 32$.

In the other direction, i.e., for large $n$, we use the following estimates, obtained from contemporary bounds on the values of Laguerre polynomials by Krasikov \cite{MR2168916}, \cite{MR2287374}.  We are grateful to the anonymous referee for improving our original estimate \eqref{Lag_est3} and we also use the proof provided by the referee for this particular estimate.
\begin{lemma}\label{lem:Lag_est}
    For $n \geq 6$, the following estimates are simultaneously true.
\begin{equation}\label{Lag_est1}
  \mathcal{L}_n(-\pi(n+1)) \le  2^n e^{\pi(n+1)}.
\end{equation}
\begin{equation}\label{Lag_est2}
  |\mathcal{L}_n(\pi(n+1))| \leq e^{\pi(n+1)/2} \sqrt{\frac{2}{n}}.   
\end{equation}
\begin{equation}\label{Lag_est3}
  |\mathcal{L}_n(2\pi(n+1))| \leq e^{\pi(n+1)} \sqrt{\frac{2}{\pi}} \, e^{-0.28(n+1)}.
\end{equation}

\end{lemma}
\begin{proof}
    The proof of~\eqref{Lag_est1} is simple:
    \begin{equation*}
      \mathcal{L}_n(-\pi(n+1)) = \sum_{k=0}^n \binom{n}{k} \frac{(\pi(n+1))^k}{k!} \le \binom{n}{\left \lfloor{\frac{n}{2}}\right \rfloor} e^{\pi(n+1)} \le 2^n e^{\pi(n+1)},
      \quad \forall n \in \N.
    \end{equation*}
    
    To prove~\eqref{Lag_est2}, we invoke the following estimate from \cite{MR2168916} (see also \cite[Thm.~1]{MR2287374}): for $n \ge 2$ and $x \in [q^2,s^2]$, where $q = \sqrt{n+1} - \sqrt{n}$ and $s = \sqrt{n+1} + \sqrt{n}$, the inequality 
    \begin{equation}\label{Kras1}
        |\mathcal{L}_n(x)| \le e^{x/2} \sqrt{ \frac{s^2-q^2}{(x-q^2)(s^2 - x)}}
    \end{equation}
    holds true. Note that $x = \pi (n+1)$ belongs to $[q^2,s^2]$, since
    $$
        q^2 = (\sqrt{n+1} - \sqrt{n})^2 \le 2n + 1 \le \pi(n+1) \le 4n \le (\sqrt{n+1} + \sqrt{n})^2 =  s^2.
    $$
    Next, for $x = \pi(n+1)$, we estimate the value inside the square root as follows\footnote{We use $4\pi-\pi^2\geq 2.5$, $6\pi - 2\pi^2\geq -1$, and $(1- \pi)^2\leq  5$.}:
    \begin{align}
        \frac{s^2-q^2}{(x-q^2)(s^2 - x)}
        & = \frac{(\sqrt{n+1} + \sqrt{n})^2 - (\sqrt{n+1} - \sqrt{n})^2}{\left(\pi(n+1) - (\sqrt{n+1} - \sqrt{n})^2\right) \left((\sqrt{n+1} + \sqrt{n})^2 - \pi(n+1)\right)}\\
        & = \frac{ 4 \sqrt{n(n+1)} }{(4\pi-\pi^2)n^2 + (6\pi - 2\pi^2) n - (1- \pi)^2}
        \le \frac{ 4n + 2}{2.5 \ n^2 - n - 5} 
        \leq \frac{2}{n},
    \end{align}
    provided $n$ is at least 6.
    Therefore, using~\eqref{Kras1}, we arrive at the desired estimate
    $$
        |\mathcal{L}_n(\pi(n+1))| \leq e^{\pi(n+1)/2} \sqrt{\frac{2}{n}}, \quad n \geq 6.
    $$
    
    It remains to prove~\eqref{Lag_est3}. We recall that $q = \sqrt{n+1} - \sqrt{n} \le s = \sqrt{n+1} + \sqrt{n}$.  Clearly, $2 \pi (n+1) \ge s^2$. Therefore, we can use Theorem~5 from \cite{MR2287374} that asserts the estimate
    \begin{equation}\label{Lag_aux4}
        |\mathcal{L}_{n}(2\pi(n+1))| \le \frac{s}{\sqrt{2 \pi (n+1)}} \ |\mathcal{L}_n(s^2)| \ e^{\pi(n+1) - s^2/2} \ e^{-R(s^2, 2 \pi (n+1))/2}, 
    \end{equation}
    where
    \begin{equation}
        R(s^2, 2 \pi (n+1)) = \int\limits_{s^2}^{2\pi(n+1)} \frac{\sqrt{(t-q^2)(t-s^2)}}{t} \, dt.
    \end{equation}
    For brevity, we set $s^2=A(n+1)$, where $A=\frac{1}{n+1}(\sqrt{n+1}+\sqrt{n})^2 \leq 4$. Then
    \begin{align}
        R(s^2, 2 \pi (n+1))
        & = \int_{A(n+1)}^{2\pi(n+1)} \sqrt{1-\frac{q^2}{t}} \sqrt{1-\frac{s^2}{t}} \, dt\\
        & \geq \sqrt{1- \frac{q^2}{s^2}} \, \int_{A(n+1)}^{2\pi(n+1)} \sqrt{1-\frac{s^2}{t}} \, dt.
    \end{align}
    The function $t \mapsto \sqrt{1-\frac{s^2}{t}}$ is increasing for $t \geq s^2$. Since $\sqrt{2\pi A} \in (A,2\pi)$, we have
    \begin{align}
        \int_{A(n+1)}^{2\pi(n+1)} \sqrt{1-\frac{s^2}{t}} \, dt
        & \geq \int_{\sqrt{2\pi A} \, (n+1)}^{2\pi (n+1)} \sqrt{1-\frac{s^2}{\sqrt{2 \pi A} \, (n+1)}} \, dt\\
        & = \sqrt{1-\frac{\sqrt{A}}{\sqrt{2\pi}}} \, (2\pi - \sqrt{2\pi A}) (n+1).
    \end{align}
    Since $A \leq 4$, we have
    \begin{equation}\label{eq:est_A}
        \sqrt{1-\frac{\sqrt{A}}{\sqrt{2\pi}}} \, (2\pi - \sqrt{2\pi A}) \geq 2 \sqrt{1-\frac{\sqrt{2}}{\sqrt{\pi}}} \, (\pi - \sqrt{2\pi}) > 0.57.
    \end{equation}
    As $q^2 s^2=1$ and $s^2$ is at least $\sqrt{2}+1$, it holds that
    \begin{equation}\label{eq:est_s4}
        \sqrt{1-\frac{q^2}{s^2}} = \sqrt{1-\frac{1}{s^4}} \geq \sqrt{1-\left(\frac{1}{\sqrt{2}+1}\right)^4} \geq 0.985.
    \end{equation}
    Combining \eqref{eq:est_A} and \eqref{eq:est_s4}, we conclude
    \begin{equation}\label{R_est}
        R(s^2,2\pi(n+1)) \geq 0.985 \cdot 0.57 (n+1) > 0.56 (n+1), \quad \forall n \in \N.
    \end{equation}
    Now, we return to relation~\eqref{Lag_aux4}. Using the classical Szeg\H{o} estimate~\eqref{eq:Szego}, and estimate \eqref{R_est}, we obtain our desired inequality:
    \begin{equation}\label{Lag_aux5}
        |\mathcal{L}_{n}(2\pi(n+1))| \le  e^{\pi(n+1)} \frac{\sqrt{4(n+1)}}{\sqrt{2 \pi (n+1)}} e^{-0.28(n+1)} \le e^{\pi(n+1)} \frac{\sqrt{2}}{\sqrt{\pi}} \, e^{-0.28(n+1)}, \quad \forall n \in \N.
    \end{equation}
\end{proof}

\textbf{Proof strategy.} In Section \ref{sec:results}, we will use the Janssen test \eqref{eq:Janssen_test} to show that certain Gabor systems with Hermite functions over a square lattice form a frame. By Lemma \ref{lem:Vgg} and \eqref{eq:Vh_L} this leads to the verification that
\begin{equation}\label{eq:explicit_square_janssen}
    \sum_{(k,l) \in \Z^2} |\mathcal{L}_n(\pi \delta (k^2+l^2))| e^{-\pi \delta (k^2+l^2)/2} < 2,
\end{equation}
where $\delta$ is the density of the square lattice. The rough strategy is always to split the series for $|k|,|l|$ relatively small and the rest. After all, we have a Gaussian weight which guarantees that the contributions for $|k|,|l|$ large will be tiny. For small orders $n$, the bound of the terms with a small index will be achieved by a finite number of computations, i.e., explicitly computing values of a finite sum, while for large orders $n$ we employ the bounds based on Lemma~\ref{lem:Lag_est}. The estimates on the remainder (large index) come from the auxiliary results in this section. The splitting will either involve squares, containing a finite number of points, or computing values for layers close to the origin, as shown in Fig.~\ref{fig:strategy}.

\begin{figure}[ht]
    \includegraphics[width=.475\textwidth]{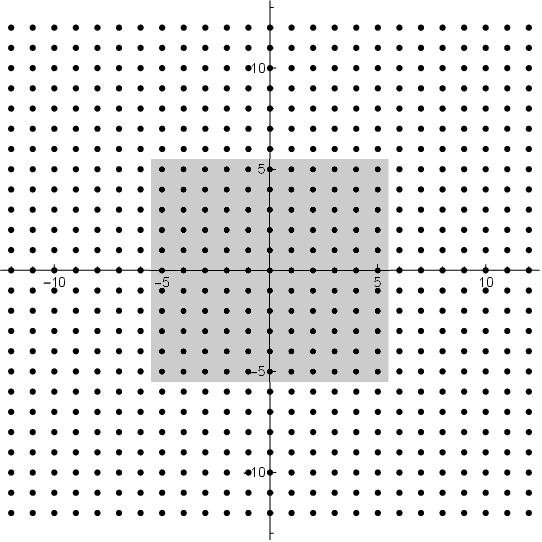}
    \hfill
    \includegraphics[width=.475\textwidth]{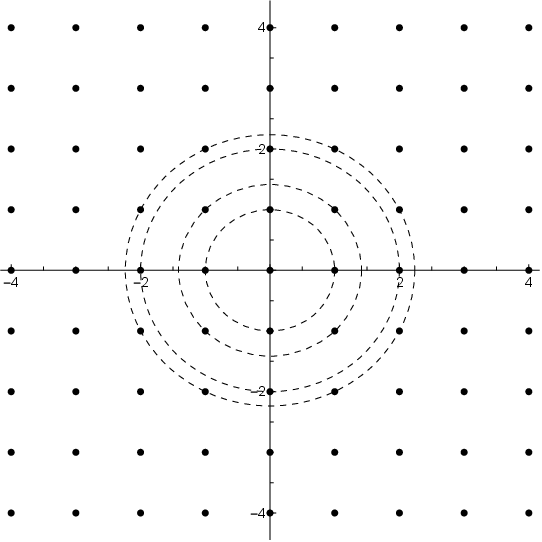}
    \caption{Left: we split $\Z^2$ into a finite section (gray area), where we perform a finite number of computations, and a complementary region, where the series can be bounded by $0 < \varepsilon \ll 1$ sufficiently small. Right: Instead of splitting $\Z^2$ by means of the max-norm, we may use the Euclidean distance and split $\Z^2$ into layers $\ell_m(\Z^2) = \{(k,l) \in \Z^2 \mid k^2+l^2 = m\}$. This also shows the connection to decomposing an integer into a sum of two squares.}
    \label{fig:strategy}
\end{figure}

A layer of a lattice $\L$ is a set of points satisfying
\begin{equation}
    \ell_r(\L) = \{ \l \in \L \mid |\l|^2 = r \}.
\end{equation}
We use the notation $|\ell_r(\L)|$ to denote the number of points in $\ell_r(\L)$. For most $r$ the set will, of course, be empty. The set with the smallest index $r > 0$ which is not empty is called the first layer (it contains the closest neighbors of the origin), the non-empty set with the second smallest index is the second layer, and so on. The decomposition of the square lattice $\Z^2$ into layers naturally leads to the question of how to write an integer as the sum of two squares. Denote by $r_2(m)$ the number of ways of writing $m \in \N_0$ as the sum of two squares, allowing zeros and distinguishing signs and order. For example,
\begin{align}
    r_2(0) & = 1 \quad \text{ as } 0 = 0^2+0^2,\\
    r_2(1) & = 4 \quad \text{ as } 1 = 0^2+1^2 = 0^2+(-1)^2 = 1^2+0^2 = (-1)^2+0^2,\\
    r_2(2) & = 4 \quad \text{ as } 2 = 1^2+1^2 = 1^2+(-1)^2 = (-1)^2+1^2 = (-1)^2+(-1)^2,\\
    r_2(3) & = 0.
\end{align}
Generally, $r_2(m) \neq 0$ if and only if all primes of the form $4k+3$ in the prime factorization of $m$ occur with an even exponent (cf.~\cite[Chap.~10, \S~3.1]{SteSha_Complex_03}). Note that, for $m \in \N_0$, $|\ell_m(\Z^2)| = r_2(m)$ and that we have the trivial estimate
\begin{equation}\label{eq:r2}
    r_2(m) \leq 4\sqrt{m} \leq 4 m, \quad m \in \N,
\end{equation}
as we have at most $\sqrt{m}$ possible choices for the pair of summands and 4 choices to place signs. This bound is quite rough, but sufficient for our purposes.

The last ingredient we need is the following estimate. We are grateful to the anonymous referee for the next lemma and its proof. This allowed us to slightly improve some of our original subsequent estimates.

\begin{lemma}\label{lem:Laguerre_Gamma_est}
    Let $a\geq 2$, $a^2\in\N$, $n\geq 1$, and $\gamma > \frac{2n+2}{a^2 \pi}$. Then 
    \begin{equation}\label{eq:lem_Laugerre_est}
        \sum\limits_{\substack{(k,l) \in \Z^2,\\
        k^2+ l^2 \geq a^2 }} (k^2+l^2)^n e^{-\pi \gamma(k^2+l^2)/2} \le \frac{4 a^{2(n+1)} e^{-\pi\gamma a^2/2}}{1-c},
    \end{equation}
    where $c = e^{-\pi\gamma/2+ (n+1)/a^2}\in(0,1)$.
\end{lemma}
\begin{proof}
    By the assumption on $\gamma$, we have $-\pi\gamma/2+(n+1)/a^2<0$, whence $c\in(0,1)$. 

    We proceed with \eqref{eq:lem_Laugerre_est}. Since the sum of squares of integers is an integer, \eqref{eq:r2} implies
    \begin{equation}\label{eq:new_lemma1}
        \begin{split}
            \sum\limits_{\substack{(k,l) \in \Z^2,\\
        k^2+ l^2 \geq a^2 }} (k^2+l^2)^n e^{-\pi \gamma(k^2+l^2)/2} 
        = \sum\limits_{m=a^2}^\infty r_2(m) m^n e^{-\pi \gamma m/2} 
        \leq  4 \sum\limits_{m=a^2}^\infty m^{n+1} e^{-\pi \gamma m/2}  .
        \end{split}
    \end{equation}
    For all $m\geq a^2$ holds
    \begin{equation}\label{eq:new_lemma2}
        \begin{split}
            \frac{(m+1)^{n+1} e^{-\pi\gamma (m+1)/2}}{ m^{n+1} e^{-\pi\gamma m/2}} 
            & = \left(1+\frac{1}{m}\right)^{n+1} e^{-\pi\gamma/2} 
            \leq \left(1+\frac{1}{a^2}\right)^{n+1} e^{-\pi\gamma/2} \\
            & \leq e^{-\pi\gamma/2+(n+1)/a^2}  = c.
        \end{split}
    \end{equation}
    We now iterate \eqref{eq:new_lemma2} $m-a^2$ times to obtain
    \begin{equation}
     \begin{split}
         m^{n+1} e^{-\pi\gamma m/2} &\leq c
        (m-1)^{n+1} e^{-\pi\gamma (m-1)/2}  \\
        & \ \vdots \\
        & \leq c^{m-a^2} (m-(m-a^2))^{n+1} e^{-\pi\gamma(m-(m-a^2))/2} \\
        & = c^{m-a^2} a^{2(n+1)} e^{-\pi\gamma a^2/2}.
     \end{split}
    \end{equation}
    Thus, \eqref{eq:new_lemma1} can be further estimated as
    \begin{equation}
        \begin{split}
            &\phantom{=}\sum\limits_{\substack{(k,l) \in \Z^2,\\
        k^2+ l^2 \geq a^2 }} (k^2+l^2)^n e^{-\pi \gamma(k^2+l^2)/2}  
        \leq 4 \sum\limits_{m=a^2}^\infty c^{m-a^2} a^{2(n+1)} e^{-\pi\gamma a^2/2} \\
        &  =  4 a^{2(n+1)} e^{-\pi\gamma a^2/2} \sum\limits_{m=0}^\infty c^{m}  = \frac{4 a^{2(n+1)} e^{-\pi\gamma a^2/2} }{1-c}
        \end{split}
    \end{equation}
    as claimed.
\end{proof}

\section{Results}\label{sec:results}
\subsection{Proof of Theorem \texorpdfstring{\ref{thm:main}}{1.1}}\label{sec:proof}
We will now prove Theorem \ref{thm:main} by settling the cases for $4 \leq n \leq 32$ first, which we solve by a finite number of computations and showing that the cut-off remainders cannot contribute enough mass such that the finite sums would violate the Janssen test in Lemma \ref{lem:Janssen_test}. The case $n \geq 33$ will be proved separately by providing uniform bounds (in $n$) on the terms in the first and second layer of the adjoint lattice and bounding the remainder.

\begin{proposition}\label{pro:4-36}
    Let $n \in \N$ with $4 \leq n \leq 32$. Then the Gabor system
    \begin{equation}
        \mathcal{G}_n = \mathcal{G} \left(h_n,\frac{1}{\sqrt{n+1}} \, \Z^2 \right)
    \end{equation}
    constitutes a frame for $L^2(\R)$.
\end{proposition}
\begin{proof}
We use the Janssen test to show that $\mathcal{G}_n$ is a frame for $n \in \N$ with $4 \leq n \leq 32$. So, we have to show that
\begin{equation}
    \sum_{(k,l) \in \Z^2} \left|\mathcal{L}_n(\pi (n+1)(k^2+l^2))\right| e^{-\pi (n+1)(k^2+l^2)/2} < 2.
\end{equation}
As $h_n$ is in $M^1(\R)$, $n \in \N_0$, \eqref{eq:cond_A} holds and the series converges unconditionally. So, we may split the series in two parts: a finite part consisting of lattice points with $|k|,|l| \leq 5$ and a remainder with $\max\{|k|,|l|\} \geq 6$.

By \eqref{eq:bound_xn}, \eqref{eq:r2}, we bound the remainder in a generous way, as we expect it to be small:
\begin{align}
    & \ \sum_{\max\{|k|,|l|\} \geq 6} \left|\mathcal{L}_n(\pi (n+1)(k^2+l^2))\right| e^{-\pi (n+1)(k^2+l^2)/2}\\
    \leq & \ (n+1) \binom{n}{\lfloor n/2 \rfloor} \sum_{\max\{|k|,|l|\} \geq 6} (\pi (n+1)(k^2+l^2))^n e^{-\pi(n+1)(k^2+l^2)/2}\\
    \leq & \ \pi^n(n+1)^{n+1} \binom{n}{\lfloor n/2 \rfloor} \sum_{k^2+l^2 \geq 36} (k^2+l^2)^n e^{-\pi(n+1)(k^2+l^2)/2}.
\end{align}
We first estimate the constant in front of the series, which is clearly increasing in $n$:
\begin{equation}
    \pi^n (n+1)^{n+1} \binom{n}{\lfloor n/2 \rfloor} \leq  \pi^{32} \cdot 33^{33} \cdot \binom{32}{16} <  10^{75},
    \quad n \leq 32.
\end{equation}
To estimate the series, we apply Lemma \ref{lem:Laguerre_Gamma_est} with $a=6$ and $\gamma = n+1$. The condition $n+1>\frac{2n+2}{6^2 \pi}$ is satisfied, and we obtain
\begin{equation}
    \sum_{k^2+l^2 \geq 36} (k^2+l^2)^n e^{-\pi (n+1)(k^2+l^2)/2} 
    \leq \frac{4\cdot 36^{n+1} e^{-18\pi(n+1)}}{1-c} 
    =  \frac{4 e^{(2\ln 6-18\pi)(n+1)}}{1-c} 
\end{equation}
with $c=  e^{-\pi (n+1)/2 +(n+1)/36}= e^{-(18\pi-1)(n+1)/36}$.
Since $c$ is monotonically decreasing in $n$, the ratio $\frac{1}{1-c}$ is also monotonically decreasing in $n$. Furthermore, $2 \ln 6-18\pi<-50$, thus we can estimate from above by evaluating the expression for $n=4$:
\begin{equation}
     \sum_{k^2+l^2 \geq 36} (k^2+l^2)^n e^{-\pi (n+1)(k^2+l^2)/2} 
    \leq \frac{4 e^{5 (2\ln 6-18\pi)}}{1- e^{-5(18\pi-1)/36}} <4\cdot 10^{-115}.
\end{equation}

Thus, for $4 \leq n \leq 33$, the remainder can be estimated by
\begin{equation}
    \sum_{\max\{|k|,|l|\} \geq 6} \left| \mathcal{L}_n(\pi(n+1)(k^2+l^2)) \right| e^{-\pi(n+1)(k^2+l^2)/2} < 4\cdot 10^{-115+75} =4\cdot  10^{-40}.
\end{equation}

Now, we are left with a finite number of computations. Using \textit{Mathematica} \cite{mathematica}, we compute the values of the finite part to 5 significant digits after the decimal point (always rounded up) to see that indeed we have
\begin{equation}\label{eq:finite_n<=36}
    \sum_{|k|,|l| \leq 5} |\mathcal{L}_n(\pi(n+1)(k^2+l^2))| \ e^{-\pi (n+1)(k^2+l^2)/2} < 1.99999 = 2 - 10^{-5}, \quad 4 \leq n \leq 32 .
\end{equation}
These calculations could be carried out to the necessary precision by hand. The precise values of \eqref{eq:finite_n<=36} for $n \in \{0, \ldots , 32\}$ are given in Fig.~\ref{fig:table_Janssen_test}. For $4 \leq n \leq 32$ they are all significantly (at least $10^{-3}$) smaller than 2. With our estimates on the remainder, this gives a complete proof of the frame property for the Gabor systems $\mathcal{G}_n$ and $4 \leq n \leq 32$.
\end{proof}

\begin{figure}[h!t]
\begin{tabular}{|r||r|r|r|r|r|r|}
    \hline
    $n$     & 0 & 1 & 2 & 3 & 4 & 5 \\
    \hline
    value   & 2.01497   & 2\phantom{.00000} & 2.00003  & 2\phantom{.00000}  & 1.99390   & 1.97889   \\
    \hline
    \hline
    $n$    & 6 & 7& 8 & 9 & 10 & 11  \\
    \hline
    value & 1.95381   & 1.91844  & 1.87308   & 1.81835           & 1.75515   & 1.68451     \\
    \hline
    \hline
    $n$    & 12 & 13 & 14 & 15 & 16 & 17 \\
    \hline
    value  & 1.60760   & 1.52567   & 1.44006 & 1.35211 & 1.26320   & 1.17470   \\
    \hline
    \hline
    $n$     & 18 & 19 & 20 & 21 & 22 & 23\\
    \hline
    value   & 1.08793   & 1.00419           & 1.07535   & 1.14951   & 1.21732   & 1.27788\\
    \hline
    \hline
    $n$     & 24 & 25 & 26 & 27 & 28 & 29 \\
    \hline
    value   & 1.33044   & 1.37440           & 1.40932   & 1.43490           & 1.45101   & 1.45770  \\
    \hline
    \hline
    $n$    & 30 & 31 & 32 &  &  & \\
    \hline
    value   & 1.45517   & 1.44376 & 1.42396   &          & & \\
    \hline
\end{tabular}
\caption{Values of the finite computations part of the Janssen test quantity for the Gabor systems $\mathcal{G}(h_n, (1/\sqrt{n+1}) \Z^2)$.
The values for $n=1$ and $n=3$ are exact and the other values are given to 5 digits after the decimal point (rounded up).}\label{fig:table_Janssen_test}
\end{figure}
As the values in Fig.~\ref{fig:table_Janssen_test} did not show a monotonic behavior, we carried out further numerical investigations for $n$ up to 120. The behavior is curious and visualized in Fig.~\ref{fig:Janssen_test}.
\begin{figure}[h!t]
    \centering
    \includegraphics[width=.65\textwidth]{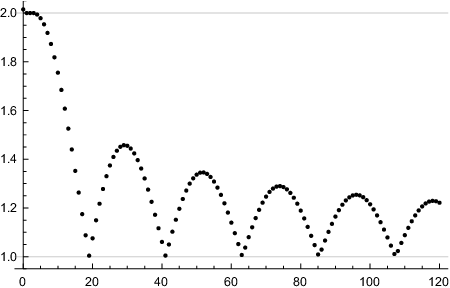}
    \caption{Values of the Janssen test quantity for $\mathcal{G}(h_n, 1/\sqrt{(n+1)} \ \Z^2)$ and $n \leq 120$.}
    \label{fig:Janssen_test}
\end{figure}

We will now treat the case $n \geq 33$. The proof, again, naturally splits into bounding the mass of the resulting series on a finite set, consisting of the origin and the first two layers of the adjoint lattice, and a small remainder. Our bounds are not the best possible and there is room to improve on $n$, i.e., make it smaller, but in any case we would be left with finitely many cases which need to be checked by hand (or computer), unless we take more layers into account for the finite set.
\begin{proposition}\label{prop_large_n}
    Let $n \geq 33$, then the Gabor system $\mathcal{G}(h_n, \frac{1}{\sqrt{n+1}} \Z \times \frac{1}{\sqrt{n+1}} \Z)$ is a frame.
\end{proposition}
\begin{proof}
    We will use the Janssen test again and show that
    \begin{equation}\label{t_jtest}
        \sum_{\l^\circ \in \L^\circ} |V_{h_n}h_n(\l^\circ)| < 2 \quad \text{ for } n \geq 33,
    \end{equation}
    which by \eqref{eq:Vh_L}, is equivalent to showing
    \begin{equation}\label{Lag_aux0}
    \sum\limits_{(k,l) \in \Z^2}
        \left| \mathcal{L}_n\left(\pi (n+1) (k^2+l^2)\right) \right| e^{-\pi(n+1)(k^2+l^2)/2}  < 2, \quad n \geq 33.
    \end{equation}
    
    Denote by $K = \left\{(0,0), (0,\pm 1), (\pm 1, 0), (\pm 1, \pm 1)\right\}$ the set which contains the origin, its 4 closest neighbors $\ell_1=\ell_1(\Z^2)$ and its 4 second closest neighbors $\ell_2=\ell_2(\Z^2)$. We split the sum at the left-hand side of~\eqref{Lag_aux0} into the sum over $K$ and $\Z^2 \setminus K$. Since $\mathcal{L}_n(0) = 1$, it suffices to prove,
    \begin{equation}
        \sum_{\ell_1} + \sum_{\ell_2} + \sum_{\ell_{3+}} < 1,
    \end{equation}
    where the abbreviated notation means we consider the sums over the first layer $\ell_1$ and the second layer $\ell_2$ and the sum over all remaining layers of $\Z^2$, abbreviated by $\ell_{3+}$ i.e.,
    \begin{align}
        \sum_{\ell_1} & = 4 |\mathcal{L}_n(\pi (n+1))| \  e^{-\pi(n+1)/2},
        \qquad \qquad
        \sum_{\ell_2} = 4 |\mathcal{L}_n(2\pi (n+1))| \ e^{-\pi(n+1)},\\
        \sum_{\ell_{3+}} & = \sum\limits_{(k,l) \in \Z^2 \setminus{K}}
        \left| \mathcal{L}_n\left( \pi (n+1)(k^2+l^2) \right) \right| \ e^{-\pi(n+1)(k^2+l^2)/2}.
    \end{align}
    By \eqref{Lag_est2} and \eqref{Lag_est3}, we have
    \begin{equation}\label{R12}
        \sum_{\ell_1} < 4\sqrt{\frac{2}{n}}
        \quad \text{and} \quad
        \sum_{\ell_2} < 4\sqrt{\frac{2}{\pi}} \, e^{-0.28(n+1)}.
    \end{equation}
    Since the coefficients of Laguerre polynomials have alternating signs, we have
    \begin{equation}\label{Lag_aux1}
        \sum_{\ell_{3+}} \leq \sum\limits_{(k,l) \in \Z^2 \setminus{K}}
        (k^2+l^2)^n \, |\mathcal{L}_n\left( -\pi (n+1) \right)|  e^{-\pi(n+1)(k^2+l^2)/2}.
    \end{equation}
    Using \eqref{Lag_est1} and the fact that $n+1 \leq \frac{\pi (n+1)}{2}$, we arrive at
    \begin{equation}
        \begin{split}
            \sum_{\ell_{3+}}
            & \leq \sum\limits_{(k,l) \in \Z^2 \setminus{K}} (k^2+l^2)^n \, 2^n e^{\pi(n+1)-\pi(n+1)(k^2+l^2)/2} \\
            & \leq \sum\limits_{(k,l) \in \Z^2 \setminus{K}} (k^2+l^2)^n \, e^{(\pi+1)(n+1)-\pi(n+1)(k^2+l^2)/2}\\
            & <  \sum\limits_{(k,l) \in \Z^2 \setminus{K}}(k^2+l^2)^n e^{-\pi(n+1)(k^2+l^2-3)/2} \\
            & = e^{3\pi (n+1)/2} \sum\limits_{(k,l) \in \Z^2 \setminus{K}} (k^2+l^2)^n e^{-\pi(n+1)(k^2+l^2)/2}.
        \end{split}
    \end{equation}
To estimate the last series in the chain of inequalities above, we use Lemma~\ref{lem:Laguerre_Gamma_est} with $\gamma = n+1$ and $a =2$:
\begin{equation}
  \sum\limits_{(k,l) \in \Z^2 \setminus{K}} (k^2+l^2)^n e^{-\pi(n+1)(k^2+l^2)/2} 
  \leq  \frac{4\cdot 2^{2(n+1)} e^{-2\pi(n+1)}}{1-e^{-\pi(n+1)/2+(n+1)/4}} 
  =\frac{4 e^{2(n+1)\ln 2-2\pi(n+1)}}{1-e^{-(2\pi-1)(n+1)/4}}.
    \end{equation}
Therefore, for all $n\geq 33$, we obtain the estimate
\begin{equation}\label{R3}
\begin{split}
    l_{3+}(n) &
\leq e^{3\pi(n+1)/2}\, \frac{4 e^{2(n+1)\ln 2-2\pi(n+1)}}{1-e^{-(2\pi-1)(n+1)/4}} 
= \frac{4e^{( -\pi+4\ln 2)(n+1)/2}}{1-e^{-(2\pi-1)(n+1)/4}} 
\leq \frac{4e^{-0.184(n+1)}}{1-e^{-(2\pi-1)(n+1)/4}}.
\end{split}
\end{equation}
    Combining~\eqref{R12} and~\eqref{R3}, we finally arrive at
    \begin{equation}
        \sum_{\ell_1} + \sum_{\ell_2} + \sum_{\ell_{3+}} < 4\sqrt{\frac{2}{n}} + 4 \sqrt{\frac{2}{\pi}} \, e^{-0.28(n+1)} + \ \frac{4 e^{-0.184(n+1)}}{1-e^{-(2\pi-1)(n+1)/4}}.
    \end{equation}
    Clearly, the expression above is decreasing in $n$. Evaluating at $n=33$ gives
    \begin{equation}
        \sum_{\ell_1} + \sum_{\ell_2} + \sum_{\ell_{3+}} < 0.985 + 0.0003 + 0.008 = 0.9933 < 1.
    \end{equation}
    So, we see that the desired bound holds for $n\geq 33$, which finishes the proof.
\end{proof}

\subsection{Failure of the Janssen test}\label{sec:failure}
\subsubsection{The case of small \texorpdfstring{$n$}{n}}
For the Gabor systems $\mathcal{G}(h_n, \delta^{-1/2} \Z^2)$ and $n \in \{0,1,2,3\}$ we know that we do not have a frame for $\delta = n+1$. In particular, the lower frame bound vanishes and the bound in the Janssen test must yield a value of at least 2. For $\delta > n+1$ we are in the safety region and have a Gabor frame. For small $n$, in particular up to 3, the Janssen test does not shed any new light on the frame set of
\begin{equation}
    \mathcal{G}_n(\delta) = \mathcal{G}(h_n, \delta^{-1/2} \Z^2).
\end{equation}
In accordance with the previous sections, we set $\mathcal{G}_n = \mathcal{G}_n(n+1)$.

For $\mathcal{G}_0$, i.e., the Gaussian function $h_0$ and the integer lattice $\Z^2$, the Janssen test fails: 
\begin{equation}
    \sum_{(k,l) \in \Z^2} |V_{h_0} h_0(k,l)| = \sum_{(k,l) \in \Z^2} e^{-\pi(k^2+l^2)/2} > \sum_{k,l \in \{-1,0,1\}} e^{-\pi(k^2+l^2)/2} > 2.004.
\end{equation}
So, the Janssen test is inconclusive, but it is a well-known fact that the lower frame bound vanishes in this case \cite{Lyu92}, \cite{Sei92}, \cite{SeiWal92} and we do not have a frame. Similarly, for $\mathcal{G}_2$ we have
\begin{align}
    \sum_{(k,l) \in \Z^2} |V_{h_2} h_2(\sqrt{3} \, k, \sqrt{3} \, l)|
    & > \sum_{(k,l) \in \{-1,0,1\}^2} \left| \mathcal{L}_2 \left(3 \pi (k^2+l^2)\right) \right| e^{-3\pi(k^2+l^2)/2}
    > 2.00001.
\end{align}
The lower frame bound vanishes, as shown by J.~Lemvig \cite[Thm.~7]{Lem16}, and we do not have a frame with the Gabor system $\mathcal{G}_2$.

Next, we consider the Gabor systems $\mathcal{G}_n$ for $n \in \{1, 3\}$. The results of Lyubarskii and Nes show that the lower bound vanishes if $n=1$. For $n=3$ we have again a result by Lemvig \cite[Thm.~8]{Lem16} which shows that the lower frame bound is zero. We remark that Boon, Zak, and Zucker \cite{BooZakZuc83} already studied the series considered in \cite{Lem16} in the context of coherent states and provided zeros (see also \cite{BooZak78}).

It is quite remarkable that for $n$ being 1 or 3 the Janssen test yields exactly 2. In this regard, these Gabor systems are closest possible to being a frame, but they still just fail.
\begin{proposition}\label{pro:Janssen_h1-3}
    Let $n \in \{1,3\}$. Then, the following hold:
    \begin{equation}
        \sum_{(k,l) \in \Z^2} |V_{h_n} h_n(\sqrt{n+1} \, k, \sqrt{n+1} \, l)| = \sum_{(k,l) \in \Z^2} \left|\mathcal{L}_n \left((n+1)\pi (k^2+l^2)\right)\right| e^{-\pi (n+1)  (k^2+l^2)/2} = 2
    \end{equation}
    and
    \begin{equation}
        \sum_{(k,l) \in \Z^2} V_{h_n} h_n(\sqrt{n+1} \, k, \sqrt{n+1} \, l) = \sum_{(k,l) \in \Z^2} \mathcal{L}_n \left((n+1)\pi (k^2+l^2)\right) e^{-\pi (n+1) (k^2+l^2)/2} = 0.
    \end{equation}
\end{proposition}
\begin{proof}
    Note that for any $n \in \N$, the term $(k,l)=(0,0)$ contributes 1. Hence, our claim is
    \begin{equation}
        \sum_{(k,l) \in \Z^2 \setminus \{(0,0)\}} \left|\mathcal{L}_n \left((n+1)\pi (k^2+l^2)\right)\right| e^{-\pi(n+1) (k^2+l^2)/2} = 1, \quad \text{ for } n = 1,3.
    \end{equation}
    By the results of Lyubarskii and Nes \cite{LyuNes13} and Lemvig \cite{Lem16} we know that  the Gabor system $\mathcal{G}(h_n, (n+1)^{-1/2} \Z^2)$ is not a frame as the lower frame bound vanishes.
    
    Now, on the one hand, we know that this implies (the Janssen test must fail) that
    \begin{equation}\label{eq:h1_3_geq_1}
        \sum_{(k,l) \in \Z^2 \setminus \{(0,0)\}} \left|\mathcal{L}_n \left((n+1)\pi (k^2+l^2)\right)\right| e^{-\pi(n+1) (k^2+l^2)/2} \geq 1.
    \end{equation}
    On the other hand, we know from the results by Janssen \cite{Jan96} (see also~\cite{Fau18}) that the minimum of the Fourier series
    \begin{equation}
        F_{g,\L}(z) = \vol(\L)^{-1}\sum_{\l^\circ \in \L^\circ} V_g g(\l^\circ) e^{2 \pi i \sigma(\l^\circ, z)}, \quad z = (x,\omega) \in \R^2,
    \end{equation}
    gives the sharp lower frame bound of the Gabor system $\mathcal{G}(g,\L)$ for $\vol(\L)^{-1} \in \N$. Thus, we have $F_{g,\L}(x,\omega) \geq 0$. Setting $\L_n = 1/\sqrt{n+1} \ \Z^2$ for the moment, this implies
    \begin{align}
        -1 \leq \vol(\L) F_{h_n, \L_n}(0,0)-1 & = \sum_{(k,l) \in \Z^2 \setminus \{(0,0)\}} V_{h_n} h_n (\sqrt{n+1} \, k, \sqrt{n+1} \, l)\\
        & = \sum_{(k,l) \in \Z^2 \setminus \{(0,0)\}} \mathcal{L}_n \left((n+1)\pi (k^2+l^2)\right) e^{-\pi(n+1) (k^2+l^2)/2}.
    \end{align}
    Note that the phase factor, imposed by \eqref{eq:Vh_L}, is one as the lattice density is $n+1$, which is even in this case. Next, we look at the Laguerre polynomials $\mathcal{L}_1(x)$ and $\mathcal{L}_3(x)$. They are
    \begin{equation}
        \mathcal{L}_1(x) = 1-x
        \quad \text{ and } \quad
        \mathcal{L}_3(x) = \frac{1}{6} \left( 6 - 18 x + 9 x^2 - x^3 \right).
    \end{equation}
    The distance of the closest lattice point to the origin is $(n+1)\pi$, so we see that $\mathcal{L}_1(x)$ and $\mathcal{L}_3(x)$ are negative when evaluated at lattice points (except for the origin). Thus, for all $(k,l) \in \Z^2 \setminus \{(0,0)\}$ we have
    \begin{equation}\label{eq:h1_3_abs_negative}
         \left|\mathcal{L}_n \left((n+1)\pi (k^2+l^2)\right)\right| e^{-\pi(n+1) (k^2+l^2)/2} =
         -\mathcal{L}_n \left((n+1)\pi (k^2+l^2)\right) e^{-\pi (n+1) (k^2+l^2)/2}.
    \end{equation}
    Hence, relation \eqref{eq:h1_3_geq_1} can equivalently be written as
    \begin{equation}
        \sum_{(k,l) \in \Z^2 \setminus \{(0,0)\}} \mathcal{L}_n \left((n+1)\pi (k^2+l^2)\right) e^{-\pi(n+1) (k^2+l^2)/2} \leq -1
    \end{equation}
    which ultimately gives the equality
    \begin{equation}
        \sum_{(k,l) \in \Z^2 \setminus \{(0,0)\}} \mathcal{L}_n \left((n+1)\pi (k^2+l^2)\right) e^{-\pi (n+1) (k^2+l^2)/2} = -1.
    \end{equation}
    The claims follow by adding the value at the origin, which contributes 1, to the sum.
\end{proof}
For the proof of this result, we made use of the fact that the considered systems do not form frames. It would be nice to find a proof which does not need this fact.

As $h_1$ remains the only case where the frame set conjecture is still open (no further restrictions than the ones from \cite{LyuNes13} are known), we considered it worthwhile to apply the Janssen test in this case. The result is that we do not get any new insights.

\begin{lemma}\label{lem:Janssen_h1}
    For $\mathcal{G}(h_1,\delta^{-1/2} \Z^2)$, $\delta \geq 1$, the Janssen test quantity decreases, i.e.,
    \begin{equation}
        \delta \mapsto \sum_{\{k,l\} \in \Z^2} |\mathcal{L}_1(\pi \delta (k^2+l^2))| e^{-\pi \delta (k^2+l^2)/2}
    \end{equation}
    is a decreasing function in $\delta$.
\end{lemma}
\begin{proof}
    We introduce the function $j_1$, which we define, for $\delta \geq 1$, to be
    \begin{equation}
        j_1(\delta) = \sum_{(k,l) \in \Z^2} |\mathcal{L}_1(\pi \delta (k^2+l^2))| e^{-
        \pi\delta (k^2+l^2)/2} = \sum_{(k,l) \in \Z^2} |1-\pi \delta (k^2+l^2)| e^{-\pi \delta (k^2+l^2)/2}.
    \end{equation}
    Since $1-\pi\delta(k^2+l^2)<0$ for $k,l\in\Z$, $(k,l)\neq (0,0)$, we have 
    \begin{equation}
        j_1(\delta) = 1 + \sum_{(k,l) \in \Z^2 \setminus \{(0,0)\}} (\pi \delta (k^2+l^2) - 1) e^{-\pi \delta (k^2+l^2)/2},
    \end{equation}
    which is clearly smooth on $(1,\infty)$ and continuous at 1. Taking the derivative, we obtain
    \begin{align}
        j_1'(\delta) & = \sum_{(k,l) \in \Z^2 \setminus \{(0,0)\}} \left( \pi (k^2+l^2)e^{-\pi \delta (k^2+l^2)/2} -\frac{\pi (k^2+l^2)}{2} (\pi \delta (k^2+l^2)-1) e^{-\pi\delta (k^2+l^2)/2} \right)\\
        & = \frac{\pi}{2} \sum_{(k,l) \in \Z^2 \setminus \{(0,0)\}} (k^2+l^2) \left(3 - \pi \delta (k^2+l^2) \right) e^{-\pi \delta (k^2+l^2)/2} < 0.
    \end{align}
\end{proof}
Proposition \ref{pro:Janssen_h1-3} together with Lemma \ref{lem:Janssen_h1} in particular shows that the Janssen test cannot be used to enlarge the frame set of $h_1$ for Gabor systems of the form $\mathcal{G}(h_1, \delta^{-1/2} \Z^2)$ for $1 < \delta < 2$. We remark that it also seems to be true that the Janssen test is inconclusive for rectangular Gabor systems of the form $\mathcal{G}(h_1, \delta^{-1/2} (a \Z \times (1/a) \Z))$, $a \in \R_+$ and $1 < \delta < 2$.

\subsection{Outside the safety region}\label{sec:outside}

In this section, we show that the frame set of the Hermite function $h_n$ on the main diagonal $a=b$ may have a rather surprising structure. In particular, $\mathfrak{F}(h_n)$ may contain a large segment on the main diagonal outside of the safety region. Our results do not rule out the option that all pairs $(a,a)$, $a<1$, belong to $\mathfrak{F}(h_n)$, $n \geq 4$, except for the points ruled out by Lemvig \cite[Thm.~7, Thm.~8]{Lem16} and Lyubarskii and Nes \cite{LyuNes13}, but we do not dare to come up with a new conjecture.

\subsubsection{The 15-th Hermite function}
For illustrative reasons, we decided to pick the particular case $n = 15$. We prove the following result.
\begin{proposition}\label{prop_15}
    The Gabor system $\mathcal{G}(h_{15}, 1/\sqrt{\delta} \ \Z^2)$ is a frame for $L^2(\R)$ for $\delta \geq 11$.
\end{proposition}
\begin{proof}
    We follow our approach in the proof of Theorem~\ref{thm:main} and use the Janssen test with sharper estimates on the Laguerre polynomials for the particular case $n = 15$.
    By Lemma~\ref{lem:Janssen_test} and \eqref{eq:Vh_L}, it suffices to show that
    \begin{equation}\label{Lag15_main}
        \sum\limits_{(k,l) \in \Z^2}
        \left| \mathcal{L}_{15}\left(\pi \delta (k^2+l^2)\right) \right| e^{-\pi \delta(k^2+l^2)/2}  < 2 \quad \text{for  } 11 \le \delta \le 16.
    \end{equation}
    We split the summation and, for $\varepsilon \in (0,1)$, reduce the problem to the verification of
    \begin{align}
        4 |\mathcal{L}_{15}(\pi \delta)| e^{- \pi \delta/2} +
        4 |\mathcal{L}_{15}(2\pi \delta)| e^{- \pi \delta} + 
        4 |\mathcal{L}_{15}(4\pi \delta)| e^{- 2\pi \delta} +
        8 |\mathcal{L}_{15}(5\pi \delta)| e^{- 5\pi \delta/2} & < 1-\varepsilon, \\
        \sum\limits_{k^2+l^2 \ge 8 }
        \left| \mathcal{L}_{15}\left(\pi \delta (k^2+l^2)\right) \right| e^{-\pi \delta(k^2+l^2)/2} & < \varepsilon. 
    \end{align}
    First, we deal with the tail estimate. By~\eqref{eq:bound_xn}, we have
    \begin{equation}
        \sum\limits_{k^2+l^2 \ge 8 }
        \left| \mathcal{L}_{15}\left(\pi \delta (k^2+l^2)\right) \right| e^{-\pi \delta(k^2+l^2)/2} 
     \le 16 \binom{15}{ 7 } \pi^{15} \delta^{15} \sum\limits_{k^2+l^2 \geq 8} (k^2+l^2)^{15}  e^{- \pi \delta (k^2+l^2)/2}.
    \end{equation}
    Next, we apply Lemma \ref{lem:Laguerre_Gamma_est} with $\gamma = \delta$, $a = \sqrt{8}$, and $n=15$,
    and get
    \begin{equation}
        \sum\limits_{k^2+l^2 \geq 8} (k^2+l^2)^{15}  e^{- \pi\delta (k^2+l^2)/2} \le \frac{4\cdot 8^{16} e^{-4\pi\delta}}{1-e^{-\pi \delta/2 +2}}.
    \end{equation}
    Since for all $11\leq \delta \leq 16$ holds 
    $$
        16 \binom{15}{7} \pi^{15}\delta^{15} \leq 16 \binom{15}{7} \pi^{15} 16^{15}  \le  10^{31} \quad \text{and} \quad
        \frac{4\cdot 8^{16} e^{-4 \pi \delta}}{1-e^{-\pi \delta/2 +2} } \leq \frac{4\cdot8^{16} e^{-44 \pi }}{1-e^{-11\pi/2 +2} } < 10^{-44},
    $$
    we deduce that
    \begin{equation}\label{15_tail}
        \sum\limits_{k^2+l^2 \ge 8 }
        \left| \mathcal{L}_{15}\left(\pi \delta (k^2+l^2)\right) \right| e^{-\pi \delta (k^2+l^2)/2}  
        \leq 
        10^{-13}.
    \end{equation}
    For the remaining layers we use the following estimates for $\delta \in [11, 16]$:
    $$
        |\mathcal{L}_{15}(\pi \delta)e^{- \pi \delta/2}| \le 0.15, \quad |\mathcal{L}_{15}(2\pi \delta) e^{- \pi \delta}| \le 0.04,
    $$
    $$
        |\mathcal{L}_{15}(4\pi \delta) e^{- 2\pi \delta}| \le 10^{-10},  \quad |\mathcal{L}_{15}(5\pi \delta) e^{-5\pi \delta/2}| \le 10^{-16}.
    $$
    We support these bounds with Fig.~\ref{fig:h15_11-16}. However, they can be verified by hand, as $\mathcal{L}_{15}$ is a fixed polynomial and evaluated on a compact interval, and its roots (given, e.g., in \cite{SalZuc49}) and critical points can be calculated with any required precision.

    We have
    $$
    \sum\limits_{k^2+l^2 < 8 }
        \left| \mathcal{L}_{15}\left(\pi \delta (k^2+l^2)\right) \right| e^{-\pi \delta (k^2+l^2)/2}  \le 1 + 4 \cdot 0.15 + 4 \cdot 0.04 + 4 \cdot 10^{-10} + 8 \cdot 10^{-16} < 1.8.
    $$
    Combining these estimates with~\eqref{15_tail}, we arrive at the desired result.
\end{proof}
\begin{figure}[h!t]
    \subfigure[$m=1$]{
        \includegraphics[width=.45\textwidth]{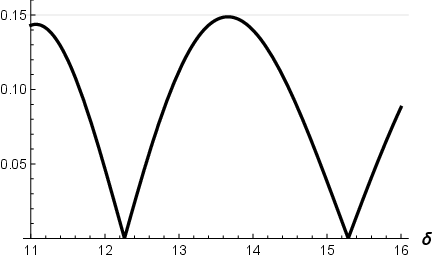}
    }
    \hfill
    \subfigure[$m=2$]{
        \includegraphics[width=.45\textwidth]{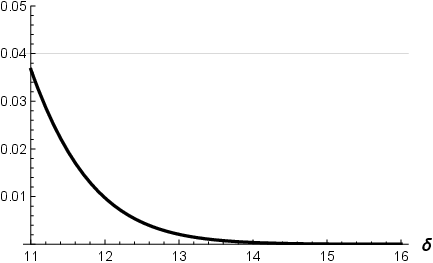}
    }
    \newline
    \subfigure[$m=4$]{
        \includegraphics[width=.475\textwidth]{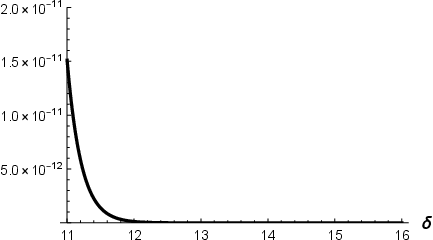}
    }
    \hfill
    \subfigure[$m=5$]{
       \includegraphics[width=.475\textwidth]{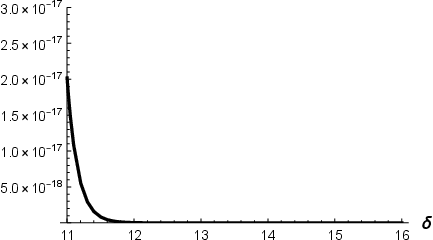}
    }
    \caption{The functions $|\mathcal{L}_{15}(m \pi \delta)|e^{-m \pi \delta/2}$ for $m=1,2,4,5$, and $\delta\in[11,16]$.} \label{fig:h15_11-16}
\end{figure}
\begin{remark}
Numerical investigations {\rm(}see Fig.~\ref{fig:h15_square}{\rm)} suggest that the Gabor system constitutes a frame in $L^2(\R)$ in a neighborhood of $\delta=3.06$, which is very far from the safety region. This can be checked with the methods used in Proposition~\ref{pro:4-36}. We have
\begin{equation}
    \sum_{|k|,|l| \leq 5} |\mathcal{L}_{15}(3.06\pi (k^2+l^2))| e^{-3.06  \pi  (k^2+l^2)/2} \leq 1.96933
\end{equation}
and the remainder term is again small due to Lemma \ref
{lem:Laguerre_Gamma_est} with $n=15$, $a=6$ and $\gamma = 3.06$:
\begin{align}
    &\phantom{=} \sum_{\max\{|k|,|l|\} \geq 6} |\mathcal{L}_{15}(3.06 \ \pi (k^2+l^2))| e^{- 3.06\pi(k^2+l^2)/2}\\
    & \leq 16 \binom{15}{7} (3.06 \pi)^{15}\sum\limits_{k^2+l^2\geq 36} (k^2+l^2)^{15} e^{- 3.06 \pi (k^2+l^2)/2}\\
    & < 6\cdot 10^{19} \cdot 3\cdot 10^{-50}< 10^{-29} .
\end{align}
Hence, the Gabor system $\mathcal{G}(h_{15}, 1/\sqrt{3.06} \ \Z^2)$ is a frame and there is an open neighborhood $U$ around $(1/\sqrt{3.06},1/\sqrt{3.06})$ such that $\mathcal{G}(h_{15}, a_U\Z \times b_U \Z)$ is a frame for $(a_U,b_U) \in U$.
\begin{figure}[h!t]
    \includegraphics[width=.475\textwidth]{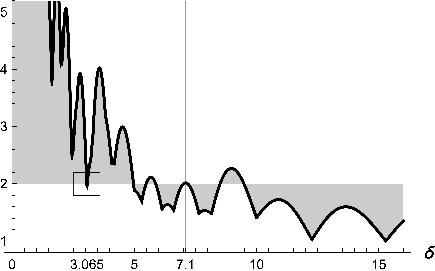}
    \hfill
    \includegraphics[width=.475\textwidth]{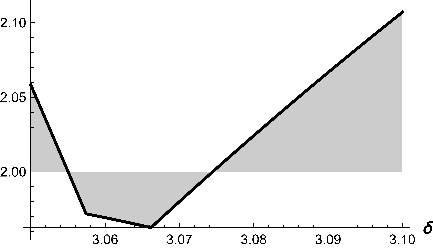}
    \caption{The Janssen test quantity for the Gabor system $\mathcal{G}(h_{15}, \delta^{-1/2} \Z^2)$. For the square lattice of density $\delta > 16$, we are in the safety region. However, the Janssen test quantity numerically provides a sufficient criterion way earlier, e.g., from $\delta \geq 10$. It is inconclusive around $\delta = 7.1$, but suffices around $\delta = 3.06$ (right we zoomed into the box from the left). The gray area is for visual aid where the value of the test is above or below 2.}
    \label{fig:h15_square}
\end{figure}
\end{remark}

\begin{remark}
    As can be seen from Fig.~\ref{fig:h15_square}, the safety region for $h_{15}$ and $\delta^{-1/2} \Z^2$ can be extended to, say, $\delta \geq 9.6$. As the frame set of a function in $M^1(\R)$ is open \cite{Gro14}, this also means that there is a neighborhood around the pairs $(1/\sqrt{\delta}, 1/\sqrt{\delta})$ for $\delta \geq 9.6$ which belongs to $\mathfrak{F}(h_{15})$. The above estimates, however, need refinements for $\delta \in [9.6, 11]$ as $\mathcal{L}_{15}(\delta \pi)$ has a zero close to $\delta=10$: $\mathcal{L}_{15}(x)$ has a root at $x = 31. 407519169754 \ldots$ \cite{SalZuc49} which gives $\delta \approx 9.99732$. The mass of the first layer is small, while the second layer contributes more mass. By dividing the interval $[9.6,11]$ into a finite number of sub-intervals and performing estimates similar to the ones above, we get the desired extension. Moreover, with a finite number of computations, it seems possible to enlarge the safety region for $h_{15}$ for rectangular lattices $a \Z \times b \Z$ to, say, $ab < 1/11$, for $a,b \in [1,1.3]$, as shown in Fig.~\ref{fig:safety_h15}. Generally, we may use the test for rectangular (or general) lattices. However, the test will often be inconclusive and will not yield new knowledge on the frame set. In particular, consider the Gabor system $(h_0, \delta^{-1/2} (a \Z \times (1/a) \Z))$. The Janssen test quantity is
    \begin{equation}
        \sum_{(k,l)\in \Z^2} e^{-\frac{\pi \delta}{2} \left(\frac{k^2}{a^2}+a^2 l^2\right)} = \theta_3\left( \frac{\delta}{2 a^2}\right) \theta_3\left( \frac{\delta a^2}{2}\right).
    \end{equation}
    Here, we use the notation
    \begin{equation}
        \theta_3 (x) = \sum_{k \in \Z} e^{-\pi k^2 x} = 1 + 2 \sum_{k \geq 1} e^{-\pi k^2 x}, \quad x > 0.
    \end{equation}
    It is not hard to see that $\lim_{x \to \infty} \theta_3(x) = 1$. We also have the functional equation
    \begin{equation}
        \sqrt{x} \ \theta_3(x) = \theta_3\left(\frac{1}{x}\right),
    \end{equation}
    which follows from the Poisson summation formula and is a special case of the Jacobi identity for the theta function. Thus, it follows that, for any fixed density $\delta > 0$ (no restrictions), the Janssen test quantity behaves like $\mathcal{O}(a)$ for rectangular lattices and $h_0$, showing that this method is clearly limited. Similar obstructions will hold for $h_n$.
\end{remark}
\begin{figure}[h!t]
    \includegraphics[width=.65\textwidth]{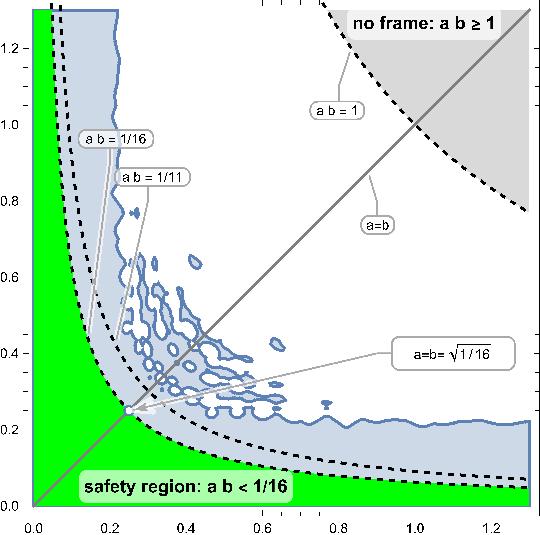}
    \caption{The safety region for $h_{15}$ and a numerical extension provided by the Janssen test quantity. For say $a,b \in [0,1.3]$, which are the depicted parameters, parts of this extension should be verifiable by hand, e.g., for $a b \leq 1/11$.}
    \label{fig:safety_h15}
\end{figure}

\subsubsection{The 9-th Hermite function}
In \cite{HorLemVid23}, the existence of pairs $(a,b)$ with $ab = 1/3$ which do not belong to the frame set of $h_n$, $n \geq 3$ was shown. We consider the example of $h_9$ and the square lattice of density 3, i.e., $1/\sqrt{3} \ \Z^2$ and show that $\mathcal{G}(h_9,1/\sqrt{3} \Z^2)$ is a frame.
\begin{proposition}\label{pro:h9}
    The Gabor system $\mathcal{G}(h_9,1/\sqrt{3} \ \Z^2)$ is a frame.
\end{proposition}
\begin{proof}
    We start with a finite number of computations, which gives
    \begin{equation}
        \sum_{|k|,|l| \leq 5} \mathcal{L}_9(3 \pi (k^2+l^2)) e^{- 3 \pi (k^2+l^2)/2} \leq 1.76496.
    \end{equation}
    The bound on the remainder is derived analogously to Proposition~\ref{pro:4-36} using Lemma \ref{lem:Laguerre_Gamma_est} applied with $n=9$, $\gamma=3$, $a^2= 36$. We estimate
    \begin{align}
        \sum_{\max\{|k|,|l|\} \geq 6} \mathcal{L}_9(3 \pi (k^2+l^2)) e^{- 3 \pi (k^2+l^2)/2}
        & \leq \sum_{\max\{|k|,|l|\} \geq 6} 10 \binom{9}{4} (3 \pi (k^2+l^2))^9 e^{- 3 \pi (k^2+l^2)/2}\\
        & < 10^{12-57} = 10^{-45}.
    \end{align}
\end{proof}
For $h_0$ and $h_1$ we know that density 3 is sufficient as we are in the safety region \cite{GroLyu07}. For $h_2$ we know that $1/\sqrt{3} \ \Z^2$ does not provide a frame \cite{Lem16}. Curiously, the only cases $n \geq 3$ where the Janssen test seems to imply that $\mathcal{G}(h_n, 1/\sqrt{3} \ \Z^2)$ is a frame (see Fig.~\ref{fig:density3}), are $n=4$ (where the value does not exceed 1.59338) and $n=9$, which we just proved. Due to \cite[Thm.~7, Thm.~8]{Lem16}, only Hermite functions of order $4m$ and $4m+1$ can constitute a frame with a square lattice of density 3, in particular $\mathcal{G}(h_{15}, 1/\sqrt{3} \ \Z^2)$ is not a frame.
\enlargethispage{\baselineskip}
\begin{figure}[h!t]
    \centering
    \includegraphics[width=.65\textwidth]{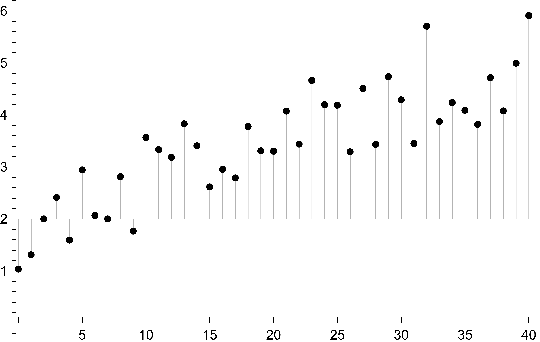}
    \caption{Values of the Janssen test quantity for $\mathcal{G}(h_n, 1/\sqrt{3} \ \Z^2)$, $n=0, \ldots , 40$. Only for $n \in \{0,1,4,9\}$ the Janssen test suffices to see that we have a frame.}
    \label{fig:density3}
\end{figure}


\end{document}